\documentclass[12pt]{amsart}

\usepackage{bm,color, ytableau}
\newtheorem{theorem}{Theorem}[section]
\newtheorem{proposition}[theorem]{Proposition}
\newtheorem{lemma}[theorem]{Lemma}

\theoremstyle{definition}

\newtheorem{example}[theorem]{Example}

\newtheorem{remark}[theorem]{Remark}

\begin{document}

\title[Reducedness and Gr\"obner bases of Specht ideals]{A note on the reducedness and Gr\"obner bases of Specht ideals}

\author{Satoshi Murai}
\address{
Satoshi Murai,
Department of Mathematics
Faculty of Education
Waseda University,
1-6-1 Nishi-Waseda, Shinjuku, Tokyo 169-8050, Japan}
\email{s-murai@waseda.jp}

\author{Hidefumi Ohsugi}
\address{
Hidefumi Ohsugi,
Department of Mathematical Sciences, School of Science,
Kwansei Gakuin University, Sanda, Hyogo 669-1337, Japan
}
\email{ohsugi@kwansei.ac.jp}

\author{Kohji Yanagawa}
\address{
Kohji Yanagawa,
Department of Mathematics, Kansai University, Suita 564-8680, Japan.
}
\email{yanagawa@kansai-u.ac.jp}

\newcommand{\add}[1]{\ensuremath{\langle{#1}\rangle}}
\newcommand{\tab}[1]{\ensuremath{{\rm Tab}({#1})}}
\newcommand{\stab}[1]{\ensuremath{{\rm STab}({#1})}}
\newcommand{\ini}[1]{\ensuremath{{\rm in}({#1})}}
\newcommand{\Fc}{{\mathcal F}}
\newcommand{\Ib}{{\mathbf I}}
\newcommand{\Vb}{{\mathbf V}}

\begin{abstract}
The Specht ideal of shape $\lambda$, where $\lambda$ is a partition, is the ideal generated by all Specht polynomials of shape $\lambda$.
Haiman and Woo proved that these ideals are reduced and found their universal Gr\"obner bases.
In this short note, we give a short proof for these results.
\end{abstract}

\maketitle

\section{Introduction}

Let $S=K[x_1,\dots,x_n]$ be a polynomial ring over an infinite field $K$.
Let $\lambda$ be a partition of $n$ and $\tab \lambda$ the set of tableaux of shape $\lambda$.
For a tableau $T \in \tab \lambda$, its {\bf Specht polynomial} $f_T$ is the product of all $x_i-x_j$ such that $i$ and $j$ are in the same column of $T$.
The ideal $I_\lambda$ of $S$ generated by all Specht polynomials of shape $\lambda$ is called the {\bf Specht ideal} (or the {\bf Garnir ideal}) of shape $\lambda$.
Such ideals are known to be related to various topics in mathematics including combinatorial commutative algebra \cite{MW,SY,SY2}, algebra and combinatorics of subspace arrangements \cite{BCES,BGS,BPS}, graph theory \cite{deL,LL,Lo}, equivariant cohomologies of Springer fibers \cite{HW} and symmetric system of equations \cite{MRV}.
Haiman and Woo proved, in their unpublished manuscript \cite{HW}, two basic but important results on Specht ideals. They proved that these ideals are reduced and found their universal Gr\"obner bases.
In this short note, we give a short proof for these results.

We quickly explain the results of Haiman and Woo (undefined notation will be explained in the next section).
They actually consider ideals which are more general than Specht ideals.
Let $P_n$ be the set of all partitions of $n$, 
and let $\Fc$ be a lower filter of $P_n$ w.r.t.\ a dominance order $\lhd$.
We call the ideal
$I_\Fc= \sum_{\lambda \in \Fc} I_\lambda$
the {\bf Specht ideal of a filter $\Fc$}.
These ideals generalize usual Specht ideals, since $I_\lambda=I_{\{\mu \in P_n \mid \mu \unlhd \lambda\}}$. 
Indeed, it is known that $I_\lambda$ contains $I_\mu$ when $\mu \lhd \lambda$ (see \cite[Theorem 1]{MRV}). 
Consider the natural permutation action of the $n$-th symmetric group $\mathfrak S_n$ on $K^n$.
For each point $\bm a \in K^n$,
the stabilizer subgroup of $\mathfrak S_n$ for $\bm a$ by this action must be isomorphic to a Young subgroup $\mathfrak S_{\mu_1}\times \cdots \times \mathfrak S_{\mu_r}$
for some $\mu=(\mu_1,\dots,\mu_r) \in P_n$.
This partition $\mu$ is called the {\bf orbit type} of $\bm a$ and will be denoted by $\Lambda(\bm a)$.
For example, $\Lambda((4,0,2,4,2,4))=(3,2,1)$.
For $\mu \in P_n$ we define $H_\mu=\{\bm a \in K^n \mid \Lambda(\bm a)=\mu\}$. 
For any subset $\mathcal G \subset P_n$
we define the reduced ideal
$$\textstyle J_{\mathcal G}=\{f\in S\mid f(\bm a)=0 \mbox{ for all } \bm a \in \bigcup_{\mu \in \mathcal G} H_\mu \}.$$
Now we state the result of Haiman and Woo \cite{HW}.

\begin{theorem}[Haiman--Woo]
\label{maintheorem}
Let $\Fc \subset P_n$ be a nonempty lower filter w.r.t.\ the dominance order and let $G_\Fc=\{ f_T \mid T \in \bigcup_{\lambda \in \Fc}  \tab \lambda\}$.
\begin{itemize}
    \item[(1)] $I_\Fc = J_{P_n\setminus \Fc}$, in particular, $I_\Fc$ is reduced.
    \item[(2)] $G_\Fc$ is a universal Gr\"obner basis for $I_\Fc$.
\end{itemize}
\end{theorem}

\begin{remark}
Theorem \ref{maintheorem} (1) was proved by Li and Li \cite{LL} when $P_n\setminus \Fc$ is a set of the form $\{\mu \in P_n \mid \mu \unrhd (m,1^{n-m})\}$
and by Kleitman and Lov\'asz \cite{Lo} when $\Fc$ is a set of the form $\{ \lambda \in P_n \mid \lambda \unlhd (m,1^{n-m})\}$.
Also, in the latter case, the Gr\"obner basis result is due to de Loera \cite{deL}.
\end{remark}

The proof of the theorem given in \cite{HW} is based on a careful analysis of the module $M_v=I_\lambda/I_{\{\mu \in P_n \mid \mu \lhd \lambda\}}$.
We give a short proof of the theorem based on induction on the numbers of the variables.

\section{Notation}

A {\bf partition} of a positive integer $n$ is a non-increasing sequence of positive integers $\lambda=(\lambda_1,\dots,\lambda_m)$ with $\lambda_1+ \cdots+\lambda_m=n$.
We write $P_n$ for the set of all partitions of $n$.
For $\lambda=(\lambda_1,\dots,\lambda_m) \in P_n$ and $\mu =(\mu_1,\dots,\mu_l) \in P_n$,
we write $\lambda \unrhd \mu$ if $\lambda$ is equal to or larger than $\mu$ w.r.t.\ the dominance order,
that is,
$$\lambda_1+ \cdots+\lambda_k \geq \mu_1+ \cdots + \mu_k \ \ \ \mbox{ for }k=1,2,\dots,\min\{m,l\}.$$
Let  $\mathcal F$ be a subset of $P_n$.
We say that $\mathcal F$ is a {\bf lower filter} of $P_n$ if it satisfies that, for any $\lambda \in \Fc$ and $\mu \in P_n$ with $\mu \unlhd \lambda$, one has $\mu \in \mathcal F$.
We define an {\bf upper filter} of $P_n$ similarly.

\ytableausetup{centertableaux,boxsize=0.3em}
 A partition $\lambda$ is frequently represented by its Young diagram. For example $(4,2,1)$ is represented as $\ydiagram {4,2,1}$. A {\bf (Young) tableau} of shape $\lambda$ is a bijective filling of the Young diagram of  $\lambda$
by the integers in $\{1,2, \ldots, n\}$. 
For example, 
$$
\ytableausetup{mathmode, boxsize=1.3em}
\begin{ytableau}
3 & 2 & 1& 7   \\
4 & 5    \\
6 \\
\end{ytableau}
$$
is a tableau of shape $(4,2,1)$. Let $\tab \lambda$ be the set of all tableaux of shape $\lambda$.
Recall that the Specht polynomial $f_T$ of $T \in \tab \lambda$ is the product of all $x_i-x_j$ such that $i$ and $j$ are in the same column of $T$ (and $j$ is in a lower position than $i$).
For example, if $T$ is the above tableau, 
then $f_T=(x_3-x_4)(x_3-x_6)(x_4-x_6)(x_2-x_5).$ 
If $\sigma \in \mathfrak S_n$ is a column stabilizer of $T$, we have $f_{\sigma T} ={\rm sgn}(\sigma)f_T$.  
In this sense, to consider $f_T$, we may assume that $T$ is {\bf column standard}, that is, all columns are increasing from top to bottom.  If a column standard tableau $T$ is also row standard (i.e., all rows are increasing from left to right), we say $T$ is {\bf standard}. Let $\stab \lambda$ be the set of all standard tableaux of shape $\lambda$.  It is a classical result that $\{ f_T \mid T \in \stab \lambda \}$ is a basis of the vector space spanned by $\{ f_T \mid T \in \tab \lambda \}$. 
We also remark that the $\mathfrak S_n$-module 
isomorphic to this vector space is called the Specht module of $\lambda$, and fundamental in the representation  theory of  symmetric groups (especially in the characteristic 0 case).

It is clear from the definition of Specht polynomials that $f_T(a_1,\dots,a_n)=0$ if there are distinct integers $i$ and $j$ such that $a_i=a_j$ and both $i$ and $j$ are in the same column of $T$.
Using this fact, it is not hard to see the following property (see \cite[Corollary~1]{MRV}).

\begin{lemma}
\label{IVI}
Let $\lambda \in P_n$ and $T \in \tab \lambda$.
Then $f_T(\bm a)=0$ for any $\bm a \in H_\mu$ with $\mu \not \! \unlhd \lambda$.
In other words, $f_T$ belongs to $J_\Fc$ where $\Fc = \{ \mu \in P_n \mid \mu \not \! \unlhd \ \lambda \}$.
\end{lemma}

Let $<$ be a monomial order on $S$.
We write $\mathrm{in}_{<}(f)$ for the initial monomial of $f \in S$ w.r.t.~$<$.
Let $I$ be an ideal of $S$.
A finite subset $G \subset I$ is called a {\bf Gr\"{o}bner basis}
of $I$ w.r.t.~$<$ if the ideal $\langle \mathrm{in}_{<}(f) \mid 0 \ne f \in I \rangle$
is generated by $\{\mathrm{in}_{<}(g) \mid  g \in G \}$.
A finite subset $G$ of $I$ of $S$ is called a {\bf universal Gr\"{o}bner basis} of $I$ if $G$ is a Gr\"{o}bner basis of $I$ 
 w.r.t.~any monomial order on $S$.
 See, e.g., \cite[Chapter 1]{binomialideals} for the details on Gr\"{o}bner bases.

\section{A short proof of Theorem \ref{maintheorem}}

For a partition $\lambda=(\lambda_1,\dots,\lambda_m) \in P_n$
and an integer $k>0$,
we write $\lambda+\add k$ for the partition of $n+1$ obtained by rearranging the sequence $(\lambda_1,\dots,\lambda_k+1,\dots,\lambda_m)$,
where we set $\lambda+\add k=(\lambda_1,\dots,\lambda_m,1)$ when $k>m$.
For example $(4,2,2,1)+\add 3=(4,3,2,1)$.
Since $\lambda + \add j  \unlhd \lambda + \add i$ holds for any $\lambda \in P_n$ and $i \le j$,
we have the following immediately.

\begin{lemma}
\label{upper}
Let $\mathcal F \subset P_n$ be an upper filter, and let $\lambda \in P_{n-1}$.
If $\lambda + \add j \in \mathcal F$, then $\lambda + \add i \in \mathcal F$ for all $i \le j$.
\end{lemma}

For any subset $\Fc \subset P_n$,
we define
$$
\Fc_k = \{ \mu \in P_{n-1}  \mid  \mu + \add k \in \Fc\}.
$$
Note that if $\Fc$ is an upper (lower) filter, then so is $\Fc_k$.

\begin{example}
Let $\Fc=\{411,33,42,51,6\}$, where $411$ means $(4,1,1)$.
Then $\Fc_1=\{311,32,41,5\},$ $\Fc_2=\{32,41,5\}$, and $\Fc_k=\{41,5\}$ for $k \geq 3$.
\end{example}

\begin{lemma}
\label{keylemma}
Let $\Fc \subset P_n$ be an upper filter, and let $f$ be a polynomial in $J_{\Fc}$
of the form
$$
f= g_d x_n^d  + \cdots +  g_1 x_n +g_0,
$$
where $g_0,\dots,g_d \in K[x_1,\dots,x_{n-1}]$ and $g_d\ne 0$.
Then $g_0,\dots,g_d$ belong to $J_{\Fc_{d+1}}$
\end{lemma}

\begin{proof}
Let $\lambda=(\lambda_1,\dots,\lambda_m) \in \Fc_{d+1}$ and ${\bm a} \in H_\lambda$.
Suppose that $\alpha_i \in K$ appears $\lambda_i$ times in ${\bm a}$ for $i=1,2,\ldots,m$,
where $\alpha_1,\ldots,\alpha_m$ are distinct each other.
By Lemma~\ref{upper}, we have $\lambda + \add i \in \Fc$ for all $1\le i \le d+1$.

If $m< d+1$, then $\lambda + \add {d+1} = (\lambda_1, \dots, \lambda_m,1)$.
Thus, for any $\alpha \in K\setminus \{\alpha_1,\dots, \alpha_m \}$,
$({\bm a}, \alpha)$ belongs to $H_{\lambda + \add {d+1}}$, and hence $f({\bm a}, \alpha) =0$.
If $m \ge d+1$, then $({\bm a}, \alpha_i)$ belongs to $H_{\lambda + \add i}$ for each $i = 1,2,\dots,d+1$.
Thus we have $ f({\bm a}, \alpha_i) =0$ for any $i = 1,2,\dots,d+1$.

In both cases, the polynomial
$
f({\bm a}, x_n) = \sum_{k=0}^d  g_k({\bm a}) x_n^k\in K[x_n]
$
of degree $d$ has at least $d+1$ roots.
This tells that
$f({\bm a}, x_n)$ is the zero polynomial in $K[x_n]$.
Hence each $g_i({\bm a}) =0$.
Thus we have $g_i \in J_{\Fc_{d+1}}$ for all $i$, as desired.
\end{proof}

\begin{proposition}
\label{lexGB}
Let $\mathcal F \subset P_n$ be a nonempty lower filter.
Then 
$$\textstyle G_{\mathcal F}= \left\{ f_T \mid T \in \bigcup_{\lambda \in \mathcal F} \tab \lambda\right\}$$ is a Gr\"{o}bner basis of $J_{P_n \setminus \Fc}$ w.r.t.~ the lexicographic order $<_{\mathrm{lex}}$
with $x_1 < \cdots < x_n$.
\end{proposition}

\begin{proof}
We write $\mathrm{in}(f)=\mathrm{in}_{<_{\mathrm{lex}}}(f)$ for $f \in S$.
We note that, if $T$ is column standard and if the number $i$ is in the $d_i$-th row of $T$, then we have $\ini {f_T}=\prod_{i=1}^n x_i^{d_i-1}$.

Let ${\mathcal F}' = P_n \setminus \Fc$.
The proof is by induction on $n$.
Since Lemma \ref{IVI} tells $G_\Fc \subset J_{{\mathcal F}'}$,
what we must prove is that, for any polynomial $f \in J_{{\mathcal F}'}$,  ${\rm in}(f)$ is divisible by ${\rm in} (f_{T'})$ for some tableau $T'$ of shape $\lambda \in \Fc$.
The assertion is trivial when $n=1$.
Suppose $n>1$.
Let
$$
f= g_d x_n^d  + \cdots +  g_1 x_n +g_0 \in J_{{\mathcal F}'},
$$
where $g_0,\dots,g_d \in K[x_1,\dots,x_{n-1}]$ and $g_d\ne 0$.
Since $g_d \in J_{{\mathcal F}'_{d+1}}$ by Lemma~\ref{keylemma} and since $G_{P_{n-1} \setminus {\mathcal F}'_{d+1}}$ is a Gr\"obner basis of $J_{ {\mathcal F}'_{d+1}}$
by the induction hypothesis,
there is a tableau $T$ of shape $\mu$ with 
$\mu \in P_{n-1} \setminus {\mathcal F}'_{d+1}$ 
such that ${\rm in} (f_T)$ divides ${\rm in} (g_d)$.
(We note that $\Fc'_{d+1} \ne P_{n-1}$ since $g_d \ne 0$ and $J_{P_{n-1}}=\{0\}$.)

Let $\lambda = \mu + \add {d+1}$.
Consider the tableau $T' \in \tab {\lambda}$ such that the image of each $i = 1,2,\dots,n-1$ is same 
for $T$ and $T'$.
Since $\mu \notin {\mathcal F}'_{d+1}$ implies $\lambda \in  {\mathcal F}$,
we have $f_{T'} \in G_{\Fc}$.
We claim ${\rm in}(f_{T'})$ divides ${\rm in}(f)$.
The image of $n$ for $T$ must be in the $(p+1)$-th row with $p \leq d$,
which tells that 
$\ini {f_{T'}} = x_n^p  \ini {f_T}$ divides ${\rm in}(f) = x_n^d {\rm in} (g_d)$, as desired.
Thus $G_{\mathcal F}$ is 
 a Gr\"{o}bner basis of $J_{{\mathcal F}'}$.
\end{proof}

\begin{remark}
Clearly, the subset
$$\textstyle  \left\{ f_T \mid T \in \bigcup_{\lambda \in \mathcal F} \stab \lambda\right\}
\subset G_\Fc$$
is a Gr\"{o}bner basis of $J_{P_n \setminus \Fc}$ w.r.t.~the lexicographic order
with $x_1 < \cdots < x_n$. However, this is still very far from a minimal Gr\"{o}bner basis. 
For example, combining the above observation with \cite[Lemma~4.3.1]{L}, we see that 
$$\textstyle  \left\{ f_T \mid T \in \stab \mu,\, \mu \unlhd \lambda, \, \mu_1 =\lambda_1 \right\}$$
is a  Gr\"{o}bner basis of $I_\lambda$, but this is far from minimal yet (unless $\lambda= (m,1^{n-m})$). 
\end{remark}

Now we are in the position to prove Theorem \ref{maintheorem}.

\begin{proof}[Proof of Theorem \ref{maintheorem}]
Since $G_\Fc$ is a Gr\"obner basis of $ J_{P_n \setminus \Fc}$, we have
$I_\Fc = \langle G_\Fc \rangle =  J_{P_n \setminus \Fc}$.
In particular, $ I_\Fc$ is reduced, and $G_\Fc$ is a Gr\"obner basis of $ I_\Fc$
 w.r.t.~the lexicographic order with
$x_1 < \cdots < x_n$.
We claim that $G_{\Fc}$ is a Gr\"obner basis for any monomial order $<$.
By the symmetry on the variables $x_1,\ldots,x_n$,
$G_\Fc$ is a Gr\"obner basis of $ I_\Fc$ w.r.t.~any lexicographic order.
Let $<$ be any monomial order.
Assume that $x_{i_1} < \cdots < x_{i_n}$ w.r.t.~$<$.
Since each Specht polynomial $f_T$ is a product of linear forms,
we have ${\rm in}_< (f_T) = {\rm in}_{<'} (f_T)$ where
$<'$ is a lexicographic order induced by the ordering $x_{i_1} < \cdots < x_{i_n}$.
This tells that $G_\Fc$ is a Gr\"obner basis of $I_\Fc$ w.r.t.\ $<$.
\end{proof}

\begin{remark}
The assertion in Theorem \ref{maintheorem} (2) holds even for a finite field $K$.
In fact, if $K$ is finite, then Theorem \ref{maintheorem} holds if the coefficient field is an infinite field $K(t)$ of rational functions in 
a variable $t$.
Since the coefficients of each $f_T$ are $\pm 1$,
the variable $t$ will never appear
when we apply Buchberger criterion to $G_\Fc$.
\end{remark}

\bigskip

\noindent
\textbf{Acknowledgments}:
This work was started at the MFO-RIMS Tandem Workshop ``Symmetries on polynomial ideals and varieties" at the Mathematisches Forschungsinstitut Oberwolfach (MFO) and the Research Institute for Mathematical Sciences (RIMS). We thank MFO and RIMS for their kind hospitality.
The first author is partially supported by KAKENHI 21K03190.
The second author is partially supported by KAKENHI 18H01134. The third author is partially supported by KAKENHI 19K03456. We thank professors Yasuhide Numata, Kosuke Shibata, Akihito Wachi and Junzo Watanabe for useful discussions on this project.

\end{document}